\documentclass[12pt,a4paper]{article}
\usepackage{mathrsfs}
\usepackage{indentfirst}
\usepackage{amsmath,amssymb,amsfonts,amsthm,graphics,graphicx}
\usepackage{makeidx}
\usepackage{color}

\newtheorem{theorem}{Theorem}

\newtheorem{proposition}{Proposition}

\newtheorem{observation}{Observation}

\begin{document}
\title{\Large\bf Lower bounds for the spanning tree numbers of two graph
products\footnote{Supported by NSFC No.11071130.}}
\author{\small Hengzhe Li, Xueliang Li, Yaping Mao, Jun Yue
\\
\small Center for Combinatorics and LPMC-TJKLC
\\
\small Nankai University, Tianjin 300071, China
\\
\small lhz2010@mail.nankai.edu.cn; lxl@nankai.edu.cn;\\
\small maoyaping@ymail.com; yuejun06@126.com}
\date{}
\maketitle
\begin{abstract}
For any graph $G$ of order $n$, the spanning tree packing number
\emph{$STP(G)$}, is the maximum number of edge-disjoint spanning
trees contained in $G$. In this paper, we obtain some sharp lower
bounds for the spanning tree numbers of Cartesian product graphs and
Lexicographic product graphs.

\end{abstract}

\section{Introduction}
All graphs in this paper are undirected, finite and simple. We refer
to the book \cite{bondy} for graph theoretic notation and
terminology not described here. For any graph $G$ of order $n$, the
\emph{spanning tree packing number} or \emph{$STP$ number}, denoted
by $\sigma=\sigma(G)$, is the maximum number of edge-disjoint
spanning trees contained in $G$. The problem studying the $STP$
number of a graph is called the \emph{Spanning Tree Packing
Problem}. For the spanning tree packing problem, Palmer
\cite{Palmer} published a survey paper on this subject. Later, Ozeki
and Yamashita gave a survey paper on the spanning tree problem. For
more details, we refer to \cite{OY}.

With graphs considered as natural models for many network design
problems, (edge-)connectivity and maximum number of edge-disjoint
spanning trees of a graph have been used as measures for reliability
and strength in communication networks modeled as a graph (see
\cite{Cunningham, Matula}).

Graph products are important methods to construct bigger graphs, and
play key roles in the design and analysis of networks. In
\cite{Peng2}, Peng and Tay obtained the spanning tree numbers of
Cartesian products of various combination of complete graphs,
cycles, complete multipartite graphs. Note that $Q_{n}\cong
P_{2}\Box P_{2}\Box\cdots\Box P_{2}$, where $Q_n$ is the
$n$-hypercube. Let $K_{n(m)}$ denote a complete multipartite graph
with $n$ parts each of which contains exact $m$ vertices.

\begin{proposition}\cite{Palmer, Peng2}\label{pro1}
$(1)$ $\sigma(K_n\Box C_m)=\lfloor \frac{n+1}{2}\rfloor$;

$(2)$ For $2\leq n\leq m$, $\sigma(K_n\Box K_m)=\lfloor
\frac{n+m-2}{2}\rfloor$;

$(3)$ For $n$-hypercube $Q_n\cong P_{2}\Box P_{2}\Box\cdots\Box
P_{2}$, $\sigma(Q_n)=\lfloor \frac{n}{2}\rfloor$;

$(4)$ $\sigma(K_{n(m)}\Box K_r)=\lfloor \frac{nm-m+r-1}{2}\rfloor$;

$(5)$ For a cycle $C_r$ with $r$ vertices, $\sigma(K_{n(m)}\Box
C_r)=\lfloor \frac{nm-m+2}{2}\rfloor$;

$(6)$ $\sigma(K_{n(m)}\Box K_{r(t)})=\lfloor
\frac{m(n-1)+(r-1)t}{2}\rfloor$;

$(7)$ $\sigma(K_{n(m)})=\lfloor \frac{m(n-1)}{2}\rfloor$.
\end{proposition}

In this paper, we focus on general graphs and give some lower bounds
for the $STP$ numbers of Cartesian product graphs, Lexicographic
product graphs. Moreover, these lower bounds are sharp.

\section{For Cartesian product}

Recall that the \emph{Cartesian product} (also called the {\em
square product}) of two graphs $G$ and $H$, written as $G\Box H$, is
the graph with vertex set $V(G)\times V(H)$, in which two vertices
$(u,v)$ and $(u',v')$ are adjacent if and only if $u=u'$ and
$(v,v')\in E(H)$, or $v=v'$ and $(u,u')\in E(G)$. Clearly, the
Cartesian product is commutative, that is, $G\Box H\cong H\Box G$.

Let $G$ and $H$ be two connected graphs with
$V(G)=\{u_1,u_2,\ldots,u_{n_1}\}$ and
$V(H)=\{v_1,v_2,\ldots,v_{n_2}\}$, respectively. We use $G(u_j,v_i)$
to denote the subgraph of $G\Box H$ induced by the set
$\{(u_j,v_i)\,|\,1\leq j\leq n_1\}$. Similarly, we use $H(u_j,v_i)$
to denote the subgraph of $G\Box H$ induced by the set
$\{(u_j,v_i)\,|\,1\leq i\leq n_2\}$. It is easy to see
$G(u_{j_1},v_i)=G(u_{j_2},v_i)$ for different $u_{j_1}$ and
$u_{j_2}$ of $G$. Thus, we can replace $G(u_{j},v_i)$ by $G(v_i)$
for simplicity. Similarly, we can replace $H(u_{j},v_i)$ by
$H(u_j)$. For any $u,u'\in V(G)$ and $v,v'\in V(H)$, $(u,v),\
(u,v')\in V(H(u))$, $(u',v),\ (u',v')\in V(H(u'))$, $(u,v),\
(u',v)\in V(G(v))$, and $(u,v),\ (u',v)\in V(G(v))$. We refer to
$(u,v)$ as the vertex corresponding to $u$ in $G(v)$. Clearly,
$|E(G\Box H)|=|E(H)||V(G)|+|E(G)||V(H)|$.

\begin{figure}[h,t,b,p]
\begin{center}
\scalebox{0.7}[0.7]{\includegraphics{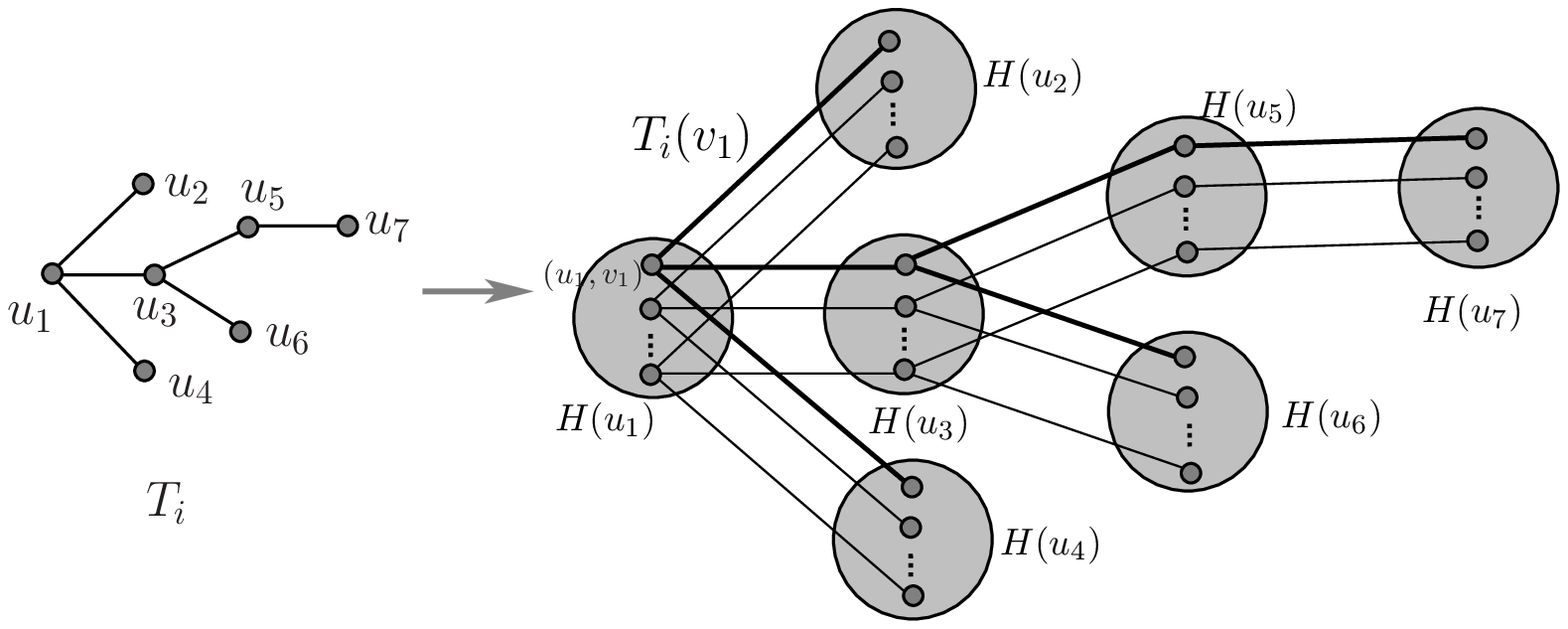}}\\
Figure 1: The parallel subgraph $\mathscr{F}_i$ in $G\Box H$
corresponding to the tree $T_i$ in $G$.
\end{center}
\end{figure}

Throughout this paper, let $\sigma(G)=k$, $\sigma(H)=\ell$, and
$T_1,T_2,\cdots,T_k$ be $k$ spanning trees in $G$ and
$T_1',T_2',\cdots,T_{\ell}'$ be $\ell$ spanning trees in $H$. For
the spanning tree $T_i \ (1\leq i\leq k)$ in $G$, we define a
spanning subgraph (see Figure $1$ for an example) of $G\Box H$ as
follows: $\mathscr{F}_i=\bigcup_{v_j\in V(H)}T_i(v_j)$, where
$T_i(v_j)$ is the corresponding tree of $T_i$ in $G(v_j)$. We call
each of $\mathscr{F}_i \ (1\leq i\leq k)$ a \emph{parallel subgraph
of $G\Box H$ corresponding to the tree $T_i$ in $G$}. For a spanning
tree $T_j'$ in $H$, we define a spanning subgraph of $G\Box H$ as
follows: $\mathscr{F}_j'=\bigcup_{u_i\in V(G)}T_j'(u_i)$, where
$T'_j(u_i)$ is the corresponding tree of $T_j' \ (1\leq i\leq \ell)$
in $H(u_i)$. We also call each of $\mathscr{F}_j' \ (1\leq i\leq
\ell)$ a \emph{parallel subgraph of $G\Box H$ corresponding to the
tree $T_j'$ in $H$}.

The following observation is helpful for understanding our main
result.

\begin{observation}\label{obs1}

Let $\mathscr{T}=\{T_1,T_2,\cdots,T_k\}$ be the set of spanning
trees of $G$, and $\mathscr{T}'=\{T_1',T_2',\cdots,T_{\ell}'\}$  be
the set of spanning trees of $H$. Then

$(1)$ $\underset{T\in \mathscr{T},T'\in \mathscr{T}'}{\bigcup} T\Box
T'\subseteq G\Box H$;

$(2)$ $E(T_i\Box T')\cap E(T_j\Box T')=\bigcup_{u\in V(G)}E(T'(u))$
for $T'\in \mathscr{T}'$ and $T_i,T_j\in \mathscr{T} \ (i\neq j)$;

$(3)$ if $G=\underset{T\in \mathscr{T}}{\bigcup}T$ and
$H=\underset{T'\in \mathscr{T}'}{\bigcup}T'$, then $\underset{T\in
\mathscr{T},T'\in \mathscr{T}'}{\bigcup} T\Box T'=G\Box H$.
\end{observation}

Let us now give our first result.

\begin{theorem}\label{th1}
For two connected graphs $G$ and $H$, $\sigma(G \Box H)\geq
\sigma(G)+\sigma(H)-1$. Moreover, the lower bound is sharp.
\end{theorem}
\begin{proof}
Let $V(G)=n_1$, $V(H)=n_2$, $\sigma(G)=k$ and $\sigma(H)=\ell$.
Since $\sigma(G)=k$, there exist $k$ spanning trees in $G$, say
$T_1,T_2,\cdots,T_{k}$. Clearly, $k\leq
\lfloor\frac{n_1}{2}\rfloor$. Since $\sigma(H)=\ell$, there exist
$\ell$ spanning trees in $H$, say $T_1',T_2',\cdots,T_{\ell}'$.
Clearly, $\ell\leq \lfloor\frac{n_2}{2}\rfloor$.

Pick up two spanning trees $T_k$ and $T_{\ell}'$ of $G$ and $H$,
respectively. Consider the graph $T_k\Box T_{\ell}'$. We will find
an our desired spanning tree of $G\Box H$ from $T_k\Box T_{\ell}'$
by a few steps.

First, we focus on the spanning tree $T_{\ell}'$ of $H$. We
successively delete some leaves of $T_{\ell}'$ to obtain a subtree
$T_{a}'$ of order $\lceil\frac{n_2}{2}\rceil$ in $T_{\ell}'$, and
the induced subgraph of all the deleted edges in $T_{\ell}'$ is a
forest, say $F_{b}'$. For example, we consider the tree $T_{\ell}'$
shown in Figure 2 $(a)$. Clearly, $|V(T_{\ell}')|=7$. We
successively delete the leaves $v_1v_4,v_6v_2,v_6v_3$, and obtain
the tree $T_a'=v_4v_5\cup v_4v_6\cup v_4v_7$ (see Figure 2 $(b)$)
and the forest $F_b'=v_1v_4\cup v_6v_2\cup v_6v_3$ (see Figure 2
$(c)$).

\begin{figure}[h,t,b,p]
\begin{center}
\scalebox{0.8}[0.8]{\includegraphics{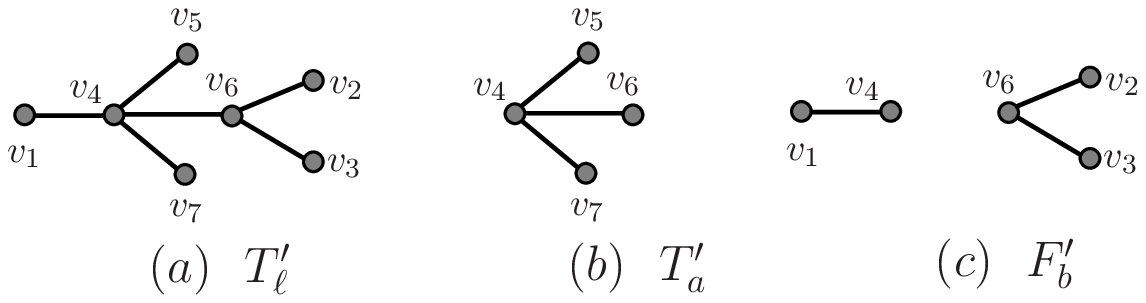}}\\
Figure 2: An example for deleting some leaves from a spanning tree
of $H$.
\end{center}
\end{figure}

Let $V(T_k)=V(G)=\{u_1,u_2,\cdots,u_{n_1}\}$. It is clear that there
are $n_1$ copies of the spanning tree $T_{\ell}'$ of $H$ in $T_k\Box
T_{\ell}'$, say
$T_{\ell}'(u_1),T_{\ell}'(u_2),\cdots,T_{\ell}'(u_{n_1})$. From the
above argument, for each $T_{\ell}'(u_i) \ (1\leq i\leq n_1)$, we
can obtain a subtree $T'_{a}(u_i) \ (1\leq i\leq n_1)$ of order
$\lceil\frac{n_2}{2}\rceil$ and a forest $F'_{b}(u_i) \ (1\leq i\leq
n_1)$ in $T_{\ell}'(u_i)$. Without loss of generality, let $u_1$ be
a root of $T_k$. Pick up $T_{\ell}'(u_1)$. Then we pick up
$\lfloor\frac{n_1-1}{2}\rfloor$ copies of $T_{a}'$ from
$T'_{\ell}(u_2),T'_{\ell}(u_3),\cdots,T'_{\ell}(u_{n_1})$, say
$T'_{a}(u_2),T'_{a}(u_3),\cdots,T'_{a}(u_{\lfloor\frac{n_1-1}{2}\rfloor+1})$,
and continue to pick up $\lfloor\frac{n_1-1}{2}\rfloor$ copies of
$F_{b}'$ from $T'_{\ell}(u_{\lfloor\frac{n_1-1}{2}\rfloor+2}),
T'_{\ell}(u_{\lfloor\frac{n_1-1}{2}\rfloor+3}),
\cdots,T'_{\ell}(u_{n_1})$, say
$F'_{b}(u_{\lfloor\frac{n_1-1}{2}\rfloor+2}),
F'_{b}(u_{\lfloor\frac{n_1-1}{2}\rfloor+3}),
\cdots,F'_{b}(u_{n_1})$.

Next, we combine
$T'_{a}(u_2),T'_{a}(u_3),\cdots,T'_{a}(u_{\lfloor\frac{n_1-1}{2}\rfloor+1})$,
$F'_{b}(u_{\lfloor\frac{n_1-1}{2}\rfloor+2}),
F'_{b}(u_{\lfloor\frac{n_1-1}{2}\rfloor+3}), \cdots,\\
F'_{b}(u_{n_1})$ with $T_{\ell}'(u_1)$ by adding some edges to form
a spanning tree of $G\Box H$ in the following way: For two trees
$T'_{\ell}(u_i)$ and $T'_{\ell}(u_j)$ such that $u_iu_j\in E(T_k)$
and $d_{T_k}(u_i,u_1)<d_{T_k}(u_j,u_1)$ (namely, $u_i$ is closer
than $u_j$ to the root $u_1$), we add some edges between
$V(T'_{\ell}(u_i))$ and $V(T'_{\ell}(u_j))$. Note that we can obtain
a subtree $T'_{a}(u_j)$ and a forest $F'_{b}(u_j)$ from
$T'_{\ell}(u_j)$. Let
$V(T_{\ell}'(u_j))=\{(u_j,v_{1}),(u_j,v_{2}),\cdots,(u_j,v_{\lceil\frac{n_2}{2}\rceil}),
(u_j,v_{\lceil\frac{n_2}{2}\rceil+1}),(u_j,v_{\lceil\frac{n_2}{2}\rceil+2}),\cdots,(u_j,v_{n_2})\}$
and $V(T'_{a}(u_j))=\{(u_j,v_{\lfloor\frac{n_2}{2}\rfloor+1}),
(u_j,v_{\lfloor\frac{n_2}{2}\rfloor+2}),\cdots,(u_j,v_{n_2})\}$ and
$V(T_{\ell}'(u_j))\setminus
V(T'_{a}(u_j))=\{(u_j,v_{1}),\\(u_j,v_{2}),\cdots,(u_j,v_{\lfloor\frac{n_2}{2}\rfloor})\}$.
If we have chosen the forest $F'_{b}(u_j)$ from $T_{\ell}'(u_j)$,
then
$E_1(u_i,u_j)=\{(u_i,v_{k})(u_j,v_{k})|\lfloor\frac{n_2}{2}\rfloor\leq
k\leq n_2\}$ is our desired edge set, which implies that we will add
the edges in $E_1(u_i,u_j)$ between $V(T_{\ell}'(u_i))$ and
$V(T_{\ell}'(u_j))$. If we have chosen the tree $T'_{a}(u_j)$ from
$T'_{\ell}(u_j)$, then $E_2(u_i,u_j)=\{(u_i,v_{k})(u_j,v_{k})|1\leq
k\leq \lfloor\frac{n_2}{2}\rfloor+1\}$ is our desired edge set,
which implies that we will add the edges in $E_2(u_i,u_j)$ between
$V(T_{\ell}'(u_i))$ and $V(T_{\ell}'(u_j))$. For the above example,
if we have chosen the forest $F'_{b}(u_j)=(u_j,v_{1})(u_j,v_{4})\cup
(u_j,v_{6})(u_j,v_{2})\cup (u_j,v_{6})(u_j,v_{3})$ from
$T'_{\ell}(u_j)$, then the edge set
$E_1(u_i,u_j)=\{(u_i,v_{4})(u_j,v_{4}),\\
(u_i,v_{5})(u_j,v_{5}),
(u_i,v_{6})(u_j,v_{6}),(u_i,v_{7})(u_j,v_{7})\}$ is our desired one
(see Figure 3 $(a)$). If we have chosen the tree
$T'_{a}(u_j)=(u_j,v_{4})(u_j,v_{5})\cup (u_j,v_{4})(u_j,v_{6})\cup
(u_j,v_{4})(u_j,v_{7})$ from $T'_{\ell}(u_j)$, then the edge set
$E_2(u_i,u_j)=\{(u_i,v_{1})(u_j,v_{1}),(u_i,v_{2})(u_j,v_{2}),
(u_i,v_{3})(u_j,v_{3}),(u_i,v_{4})(u_j,v_{4})\}$ is our desired one;
see Figure 3 $(b)$.

\begin{figure}[h,t,b,p]
\begin{center}
\scalebox{0.7}[0.7]{\includegraphics{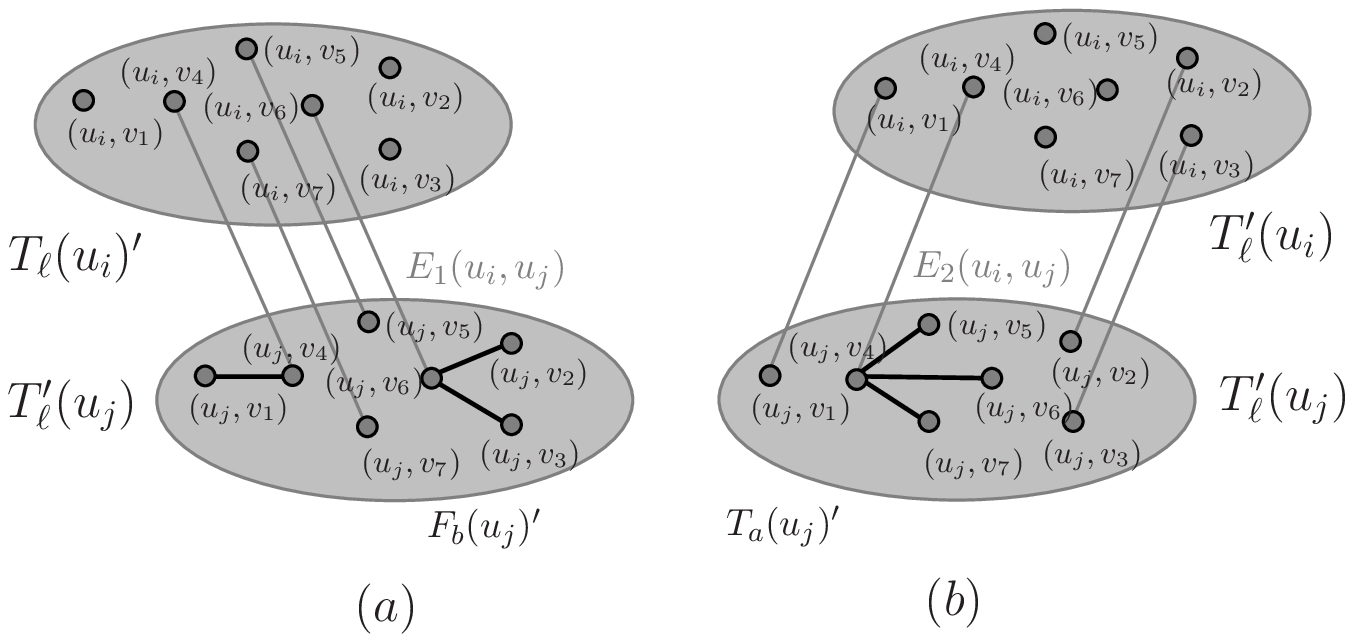}}\\
Figure 3: An example for the procedure of adding edges.
\end{center}
\end{figure}

We continue to complete the above adding edges procedure. In the
end, we obtain a spanning tree of $G\Box H$ in $T_k\Box T'_{\ell}$,
say $\widehat{T}$. An example is given in Figure $4$. Let us focus
on the graph $T_k\Box T'_{\ell}\setminus E(\widehat{T})$. In order
to form the tree $\widehat{T}$, we have used the tree
$T'_{\ell}(u_1)$, the subtrees
$T'_{a}(u_2),T'_{a}(u_3),\cdots,T'_{a}(u_{\lfloor\frac{n_1-1}{2}\rfloor+1})$
and the forests
$F'_{b}(u_{\lfloor\frac{n_1-1}{2}\rfloor+2}),F'_{b}(u_{\lfloor\frac{n_1-1}{2}\rfloor+3})$,
$\cdots,F'_{b}(u_{n_1})$ among
$T_{\ell}'(u_2),T_{\ell}'(u_3),\cdots,T_{\ell}'(u_{n_1})$. So there
are $\lfloor\frac{n_1-1}{2}\rfloor$ copies of $T'_{a}$, namely,
$T'_{a}(u_{\lfloor\frac{n_1-1}{2}\rfloor+2}),
T'_{a}(u_{\lfloor\frac{n_1-1}{2}\rfloor+3}),\cdots,T'_{a}(u_{n_1})$
in $T_k\Box T'_{\ell}\setminus E(\widehat{T})$, and there are also
$\lfloor\frac{n_1-1}{2}\rfloor$ copies of $F'_{b}$, namely,
$F'_{b}(u_{2}),
F'_{b}(u_{3}),\cdots,F'_{b}(u_{\lfloor\frac{n_1-1}{2}\rfloor+1})$ in
$T_k\Box T'_{\ell}\setminus E(\widehat{T})$. We must notice another
fact that, for two trees $T'_{\ell}(u_i)$ and $T'_{\ell}(u_j)$ such
that $u_iu_j\in E(T_k)$ and $d_{T_k}(u_i,u_1)<d_{T_k}(u_j,u_1)$
(namely, $u_i$ is closer than $u_j$ to the root $u_1$), we have used
$\lceil\frac{n_2}{2}\rceil$ edges belonging to $E_1(u_i,u_j)$ or
$E_2(u_i,u_j)$ between $V(T'_{\ell}(u_i))$ and $V(T'_{\ell}(u_j))$
and hence there are
$n_2-\lceil\frac{n_2}{2}\rceil=\lfloor\frac{n_2}{2}\rfloor$
remaining edges belonging to $\overline{E}_1(u_i,u_j)$ or
$\overline{E}_2(u_i,u_j)$ between $V(T'_{\ell}(u_i))$ and
$V(T'_{\ell}(u_j))$ in $T_k\Box T'_{\ell}\setminus E(\widehat{T})$,
where
$\overline{E}_1(u_i,u_j)=E[V(T'_{\ell}(u_i)),V(T'_{\ell}(u_j))]\setminus
E_1(u_i,u_j)$ and
$\overline{E}_2(u_i,u_j)=E[V(T'_{\ell}(u_i)),V(T'_{\ell}(u_j))]\setminus
E_2(u_i,u_j)$. Later, we will use all the above remaining edges to
form some new spanning trees of $G\Box H$.

\begin{figure}[h,t,b,p]
\begin{center}
\scalebox{0.7}[0.7]{\includegraphics{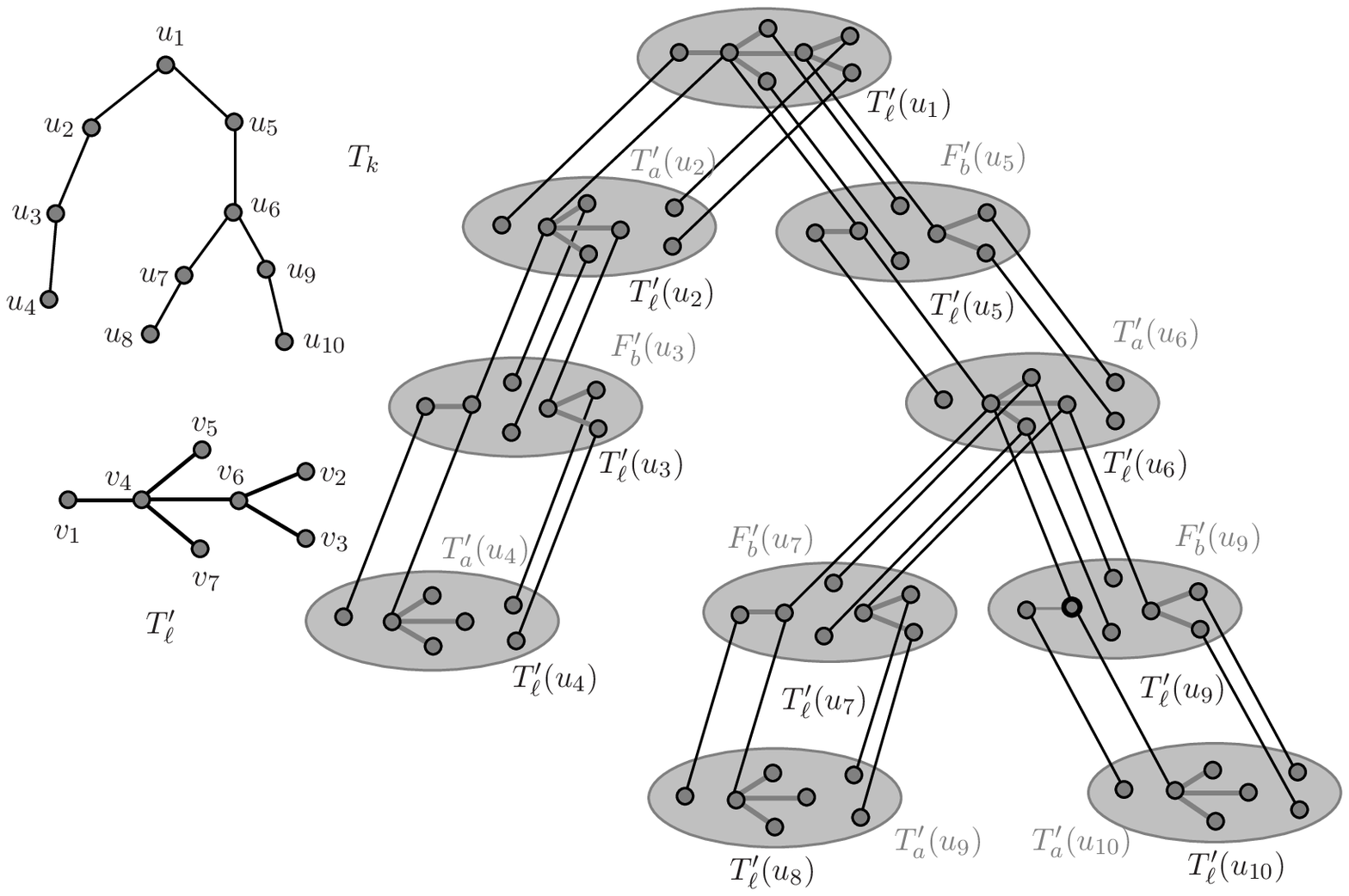}}\\
Figure 4: A spanning tree of $G\Box H$ from $T_i\Box T_j$.
\end{center}
\end{figure}

Let $\mathscr{F}_i \ (1\leq i\leq k-1)$ be the parallel subgraph of
$G\Box H$ corresponding to $T_i \ (1\leq i\leq k-1)$ in $G$. Note
that one of $\mathscr{F}_1,\mathscr{F}_1,\cdots,\mathscr{F}_{k-1}$,
one of $T'_{a}(u_{\lfloor\frac{n_1-1}{2}\rfloor+2}),
T'_{a}(u_{\lfloor\frac{n_1-1}{2}\rfloor+3}),$
$\cdots,T'_{a}(u_{n_1})$ and one of $F'_{b}(u_{2}),
F'_{b}(u_{3}),\cdots,F'_{b}(u_{\lfloor\frac{n_1-1}{2}\rfloor+1})$
can form a spanning tree of $G\Box H$. So we can obtain $k-1$
spanning trees of $G\Box H$ since $k-1\leq
\lfloor\frac{n_1}{2}\rfloor-1$.

Let $\mathscr{F}'_j \ (1\leq j\leq \ell-1)$ be the parallel subgraph
of $G\Box H$ corresponding to $T'_j \ (1\leq j\leq \ell-1)$ in $H$,
where $\mathscr{F}'_j=\bigcup_{u_i\in V(G)}T_j'(u_i)$. Note that one
of $\mathscr{F}'_1,\mathscr{F}'_2,\cdots,\mathscr{F}'_{\ell-1}$ and
one edge of $\overline{E}_1(u_i,u_j)$ or $\overline{E}_2(u_i,u_j)$
for each $u_iu_j\in V(T_k) \ (2\leq i\neq j\leq n_1)$ can form a
spanning tree of  $G\Box H$. Since $\ell-1\leq
\lfloor\frac{n_2}{2}\rfloor -1$ and $|\overline{E}_r(u_i,u_j)|\leq
\lfloor\frac{n_2}{2}\rfloor \ (r=1,2)$, we can obtain $\ell-1$
spanning trees of $G\Box H$.

In the following, we summarize all the edge-disjoint spanning trees
obtained by us.

$\bullet$ $k-1$ spanning trees of $G \Box H$ obtained from the
parallel subgraphs
$\mathscr{F}_1,\mathscr{F}_2,\cdots,\mathscr{F}_{k-1}$, the subtrees
$T'_{a}(u_{\lfloor\frac{n_1-1}{2}\rfloor+2}),
T'_{a}(u_{\lfloor\frac{n_1-1}{2}\rfloor+3}),
\cdots,T'_{a}(u_{n_1})$, and the forests $F'_{b}(u_{2}),
F'_{b}(u_{3}),\cdots,\\
F'_{b}(u_{\lfloor\frac{n_1-1}{2}\rfloor+1})$;

$\bullet$ $\ell-1$ spanning trees of $G\Box H$ obtained from the
parallel subgraphs
$\mathscr{F}_1',\mathscr{F}_2',\cdots,\mathscr{F}_{\ell-1}'$ and the
$\ell-1$ edges in $\overline{E}_1(u_i,u_j)$ or
$\overline{E}_2(u_i,u_j)$ for each $u_iu_j\in E(T_k) \ (1\leq i\neq
j\leq n_1)$;

$\bullet$ one spanning trees $\widehat{T}$ of $G \Box H$ obtained
from the tree $T'_{\ell}(u_{1})$, the subtrees $T'_{a}(u_{2}),
T'_{a}(u_{3}),\\
\cdots,T'_{a}(u_{\lfloor\frac{n_1-1}{2}\rfloor+1})$, and the forests
$F'_{b}(u_{\lfloor\frac{n_1-1}{2}\rfloor+2}),
F'_{b}(u_{\lfloor\frac{n_1-1}{2}\rfloor+3}),\cdots, F'_{b}(u_{n_1})$
and the edges of $E_1(u_i,u_j)$ or $E_2(u_i,u_j)$ for each
$u_iu_j\in E(T_k) \ (1\leq j\leq n_1)$.

From the above arguments, we know that there exist $k+\ell-1$
spanning trees of $G \Box H$, that is, $\sigma(G \Box H)\geq
k+\ell-1=\sigma(G)+\sigma(H)-1$.
\end{proof}

To show the sharpness of the above bound, we consider the following
examples.

\noindent \textbf{Example 1}. $(1)$ Let $G$ and $H$ be two paths of
order $n \ (n\geq 2)$. Clearly, $\sigma(G)=\sigma(H)=1$,
$|V(G)|=|V(H)|=n$. On the one hand, from the above theorem, we have
$\sigma(G \Box H)\geq \sigma(G)+\sigma(H)-1=1$. On the other hand,
since $|E(G\Box H)|=|E(H)||V(G)|+|E(G)||V(H)|=2n(n-1)$, it follows
that $\sigma(G \Box H)\leq \lfloor\frac{2n(n-1)}{n^2-1}\rfloor=1$.
So $\sigma(G \Box H)=\sigma(G)+\sigma(H)-1$;

$(2)$ Let $G=K_{2n}$ and $H=C_{m}$. Clearly, $\sigma(G)=n$,
$\sigma(H)=1$. From $(1)$ of Proposition \ref{pro1}, $\sigma(G \Box
H)=n=\sigma(G)+\sigma(H)-1$;

$(3)$ Let $G=K_{2n}$ and $H=K_{2m}$. Clearly, $\sigma(G)=n$,
$\sigma(H)=m$. From $(2)$ of Proposition \ref{pro1}, $\sigma(G \Box
H)=n+m-1=\sigma(G)+\sigma(H)-1$;

$(4)$ Let $n$ be an odd integer, and $G=Q_{n-1}$ and $H=P_{2}$.
Clearly, $\sigma(G)=\lfloor\frac{n-1}{2}\rfloor$, $\sigma(H)=1$.
From $(3)$ of Proposition \ref{pro1},
$\lfloor\frac{n}{2}\rfloor=\sigma(Q_{n})=\sigma(G \Box
H)=\lfloor\frac{n-1}{2}\rfloor+1-1=\sigma(G)+\sigma(H)-1$.

$(5)$ Let $m,n,r$ be three odd integers such that $m,r$ are even, or
$n$ is odd and $r$ is even, and $G=K_{n(m)}$ and $H=K_{r}$. Clearly,
$\sigma(G)=\lfloor\frac{m(n-1)}{2}\rfloor$,
$\sigma(H)=\lfloor\frac{r}{2}\rfloor$. From $(4)$ of Proposition
\ref{pro1}, $\sigma(G \Box
H)=\lfloor\frac{m(n-1)+r-1}{2}\rfloor=\frac{m(n-1)+r-2}{2}$ and
hence $\sigma(G)+\sigma(H)-1=\frac{m(n-1)}{2}+\frac{r}{2}-1
=\frac{m(n-1)+r-2}{2}=\sigma(G \Box H)$.

\section{For Lexicographic product}

Recall that the \emph{Lexicographic product} of two graphs $G$ and
$H$, written as $G\circ H$, is defined as follows: $V(G\circ
H)=V(G)\times V(H)$. Two distinct vertices $(u,v)$ and $(u',v')$ of
$G\circ H$ are adjacent if and only if either $(u,u')\in E(G)$ or
$u=u'$ and $(v,v')\in E(H)$. Note that unlike the Cartesian Product,
the Lexicographic product is a non-commutative product. Thus $G\circ
H$ need not be isomorphic to $H\circ G$. Clearly, $|E(G\circ
H)|=|E(H)||V(G)|+|E(G)||V(H)|^2$.

The following observation is helpful for understanding our main
result.

\begin{observation}\label{obs2}

Let $\mathscr{T}=\{T_1,T_2,\cdots,T_k\}$ be the set of spanning
trees of $G$, and $\mathscr{T}'=\{T_1',T_2',\cdots,T_{\ell}'\}$  be
the set of spanning trees of $H$. Then

$(1)$ $\underset{T\in \mathscr{T},T'\in \mathscr{T}'}{\bigcup}
T\circ T'\subseteq G\circ H$;

$(2)$ $E(T_i\circ T')\cap E(T_j\circ T')=\bigcup_{u\in
V(G)}E(T'(u))$ for $T'\in \mathscr{T}'$ and $T_i,T_j\in \mathscr{T}
\ (i\neq j)$;

$(3)$ if $G=\underset{T\in \mathscr{T}}{\bigcup}T$ and
$H=\underset{T'\in \mathscr{T}'}{\bigcup}T'$, then $\underset{T\in
\mathscr{T},T'\in \mathscr{T}'}{\bigcup} T\circ T'=G\circ H$.
\end{observation}

From the definition, the Lexicographic product graph $G\circ H$ is a
graph obtained by replacing each vertex of $G$ by a copy of $H$ and
replacing each edge of $G$ by a complete bipartite graph
$K_{n_2,n_2}$. For an edge $u_iu_j\in E(G) \ (1\leq i,j\leq n_1)$,
the induced subgraph obtained from the edges between the vertex set
$V(H(u_i))=\{(u_i,v_1),(u_i,v_2),\cdots,(u_i,v_{n_2})\}$ and the
vertex set $V(H(u_j))=\{(u_j,v_1),(u_j,v_2),\cdots,(u_j,v_{n_2})\}$
in $G\circ H$ is a complete equipartition bipartite graph of order
$2n_2$, denoted by $K_{H(u_i),H(u_j)}$.

Laskar and Auerbach \cite{LA} obtained the following result.

\begin{proposition}\cite{LA} \label{pro2}
For all even $r\geq 2$, $K_{r,r}$ is the union of $\frac{1}{2}r$ of
its Hamiltonian cycles.
\end{proposition}

From their result, $K_{H(u_i),H(u_j)}$ can be decomposed into
$\frac{1}{2} n_2$ Hamiltonian cycles for $n_2$ even, or $\frac{1}{2}
(n_2-1)$ Hamiltonian cycles and one perfect matching for $n_2$ odd.
Therefore, $K_{H(u_i),H(u_j)}$ can be decomposed into $n_2$ perfect
matchings $M_1,M_2,\cdots,M_{n_2}$ of $K_{H(u_i),H(u_j)}$ such that
$C_i=M_{2i-1} \cup M_{2i} \ (1\leq i\leq
\lfloor\frac{n_2}{2}\rfloor)$ is a Hamiltonian cycle of
$K_{H(u_i),H(u_j)}$. We call each $C_i$ an \emph{perfect cycle}.
Furthermore, $K_{H(u_i),H(u_j)}$ can be decomposed into $x$ perfect
cycles and $n_2-2x$ perfect matchings.

Since $\tau(G)=k$, there exist $k$ spanning trees in $G$, say
$T_1,T_2,\cdots,T_{k}$. For each $T_i \ (1\leq i\leq k)$, there is
spanning subgraph $\mathscr{T}_i \ (1\leq i\leq k)$ in $G\circ H$
corresponding to the spanning tree $T_i$ in $G$; see Figure $5$. As
we know, $K_{n_2,n_2}$ can be decomposed into $n_2$ perfect
matchings. So each such spanning subgraph $\mathscr{T}_i \ (1\leq
i\leq k)$ can be decomposed into $n_2$ parallel subgraphs
corresponding to the spanning tree $T_i$ in $G$, say
$\mathscr{F}_{i,1},\mathscr{F}_{i,2},\cdots,\mathscr{F}_{i,n_2}$.
Furthermore, we can decompose each $\mathscr{T}_i \ (1\leq i\leq k)$
into $n_2$ parallel subgraphs
$\mathscr{F}_{i,1},\mathscr{F}_{i,2},\cdots,\mathscr{F}_{i,n_2}$
such that $\mathscr{F}_{i,2j-1}\cup \mathscr{F}_{i,2j} \ (1\leq
j\leq \lfloor\frac{n_2}{2}\rfloor)$ contains $n_1-1$ perfect cycles.

\begin{figure}[h,t,b,p]
\begin{center}
\scalebox{0.7}[0.7]{\includegraphics{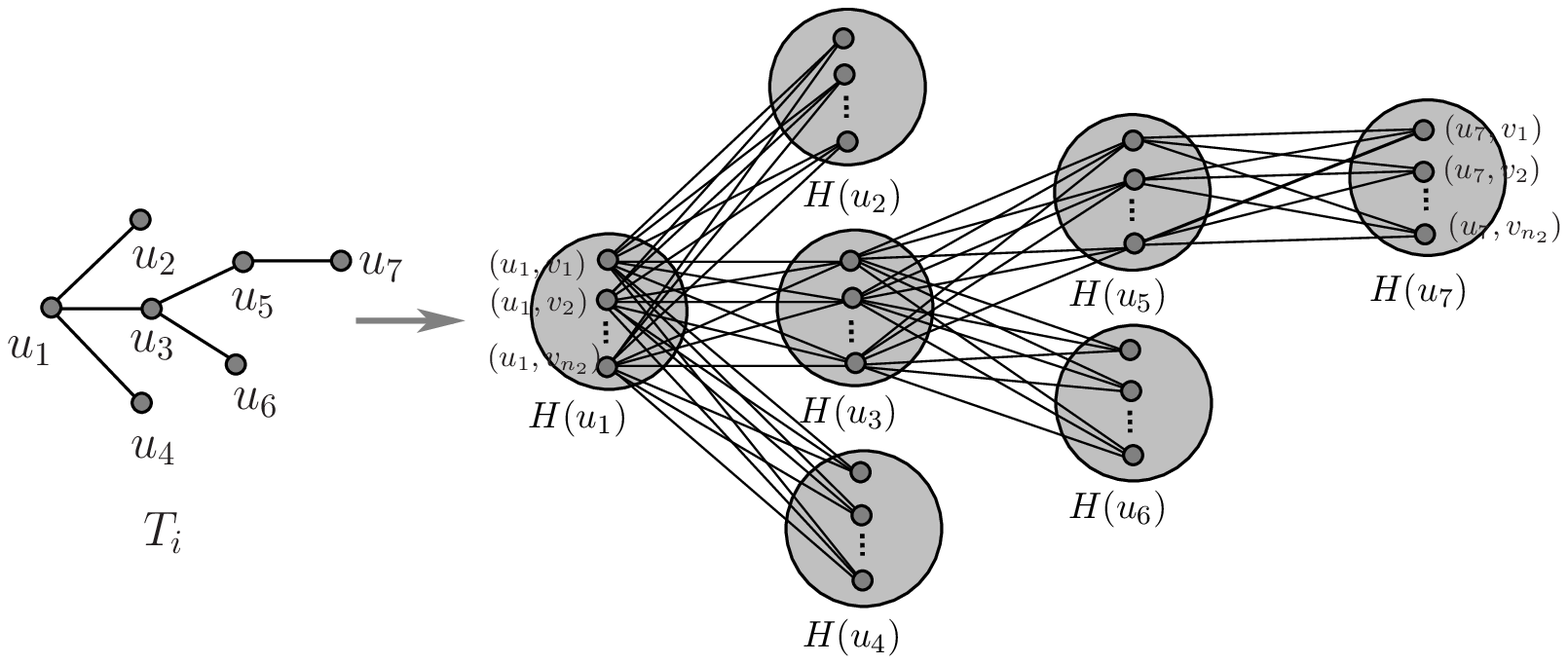}}\\
Figure 5: The spanning subgraph $\mathscr{T}_i$ in $G\circ H$
corresponding to the tree $T_i$ in $G$.
\end{center}
\end{figure}

For more clear, we give the following observation.

\begin{observation}\label{obs3}
Let $T$ be a tree, $H$ be a connected graph of order $n$. Then all
the edges of $T\circ H$ corresponding to the edges of $T$ can be
discomposed into $n$ parallel subgraphs of $T\circ H$ corresponding
to the tree $T$, say
$\mathscr{F}_{1},\mathscr{F}_{2},\cdots,\mathscr{F}_{n}$, such that
there exist $2x$ parallel subgraphs
$\mathscr{F}_{1},\mathscr{F}_{2},\cdots,\mathscr{F}_{2x}$ such that
$\mathscr{F}_{2j-1}\cup \mathscr{F}_{2j} \ (1\leq j\leq x\leq
\lfloor\frac{n_2}{2}\rfloor)$ contains exact $n_1-1$ perfect cycles.
\end{observation}

After the above preparations, we now give our result.

\begin{theorem}\label{th2}
Let $G$ and $H$ be two connected graphs. $\sigma(G)=k$,
$\sigma(H)=\ell$, $|V(G)|=n_1$, and $|V(H)|=n_2$. Then

$(1)$ if $k n_2=\ell n_1$, then $\sigma(G \circ H)\geq k n_2(=\ell
n_1)$;

$(2)$ if $\ell n_1>k n_2$, then $\sigma(G \circ H)\geq
kn_2-\lceil\frac{k n_2-1}{n_1}\rceil+\ell-1$;

$(3)$ if $\ell n_1<k n_2$, then $\sigma(G \circ H)\geq
kn_2-2\lceil\frac{kn_2-1}{n_1+1}\rceil+\ell-1$.

Moreover, the lower bounds are sharp.
\end{theorem}
\begin{proof}
$(1)$ Since $\sigma(G)=k$, there exist $k$ spanning tree in $G$, say
$T_1,T_2,\cdots,T_k$. Then there exist parallel subgraphs
$\mathscr{F}_{i,j} \ (1\leq i\leq k, 1\leq j\leq n_2)$ in $G \circ
H$ corresponding to the spanning tree $T_i \ (1\leq i\leq k)$ in
$G$. Since $\sigma(H)=\ell$, there exist $\ell$ spanning trees of
$H$, say $T'_1,T'_2,\cdots,T'_{\ell}$. Then, for a spanning tree
$T_j' \ (1\leq j\leq \ell)$ in $H$, there is parallel subgraph
$\mathscr{F}_j'=\bigcup_{u_i\in V(G)}T_j'(u_i)$ in $G \circ H$
corresponding to the spanning tree $T'_j$ of $H$, where $T'_j(u_i)$
is the corresponding tree of $T_j'$ in $H(u_i)$. So there are $\ell
n_1$ such trees $T_j'(u_i) \ (1\leq i\leq n_1, 1\leq j\leq \ell)$ in
$G \circ H$. Because one tree of $\{T_j'(u_i)|1\leq i\leq n_1, 1\leq
j\leq \ell\}$ and one of $\{\mathscr{F}_{i,j}|1\leq i\leq k, 1\leq
j\leq n_2\}$ can form a spanning tree of $G \circ H$, we can get $k
n_2=\ell n_1$ spanning trees in $G\circ H$, namely, $\sigma(G \circ
H)\geq k n_2(=\ell n_1)$.

$(2)$ Since $\sigma(G)=k$, there exist $k$ spanning tree in $G$, say
$T_1,T_2,\cdots,T_k$. Then there exist parallel subgraphs
$\mathscr{F}_{i,j} \ (1\leq i\leq k, 1\leq j\leq n_2)$ in $G \circ
H$ corresponding to the spanning tree $T_i \ (1\leq i\leq k)$ in
$G$. We pick up $\mathscr{F}_{k,n_2}$. Note that
$\mathscr{F}_{k,n_2}=\bigcup_{v_i\in V(H)}T_{k}(v_i)$, where
$T_{k}(v_i)$ is the corresponding tree of $T_{k}$ in $G(v_i) \ 1\leq
i\leq n_2$. Thus we can obtain $n_2$ such trees isomorphic to the
spanning tree $T_{k}$ of $H$ from $\mathscr{F}_{k,n_2}$, say
$T_{k}(v_1),T_{k}(v_2),\cdots,T_{k}(v_{n_2})$. Since $\tau(H)=\ell$,
there exist $\ell$ spanning tree in $G$, say
$T'_1,T'_2,\cdots,T'_{\ell}$. Then there exist parallel subgraphs
$\mathscr{F}_j'=\bigcup_{u_i\in V(G)}T_j'(u_i) \ (1\leq j\leq \ell)$
in $G \circ H$ corresponding to the spanning tree $T_j'$ in $H$,
where $T'_j(u_i)$ is the corresponding tree of $T_j'$ in $H(u_i)$.
Pick up $x$ parallel subgraphs, without loss of generality, let them
be $\mathscr{F}_1',\mathscr{F}_2',\cdots,\mathscr{F}_{x}'$. We can
obtain $xn_1$ trees $T_j'(u_i) \ (1\leq j\leq x, 1\leq i\leq n_1)$
isomorphic to the tree $T_j'$ in $H$. Note that each of
$\{\mathscr{F}_{i,j}|1\leq i\leq k, 1\leq j\leq n_2\}\setminus
\mathscr{F}_{k,n_2}$ and each of the trees $T_j'(u_i) \ (1\leq j\leq
x, 1\leq i\leq n_1)$ can form a spanning tree of $G \circ H$. If
$xn_1\geq kn_2-1$, then we can obtain $kn_2-1$ spanning trees of $G
\circ H$. Consider the remaining $\ell-x$ parallel subgraphs
$\mathscr{F}_{x+1}',\mathscr{F}_{x+2}',\cdots,\mathscr{F}_{\ell}'$.
Note that each of them and each of the trees
$T_{k}(v_1),T_{k}(v_2),\cdots,T_{k}(v_{n_2})$ can form a spanning
tree of $G \circ H$. Since $\ell-x\leq \ell\leq
\lfloor\frac{n_2}{2}\rfloor\leq n_2$, we can obtain $\ell-x$
spanning trees of $G \circ H$ and hence the total number of the
spanning trees is $(kn_2-1)+(\ell-x)$, namely, $\sigma(G \circ
H)\geq kn_2-1+\ell-x$. Since $xn_1\geq kn_2-1$, it follows that
$x=\lceil\frac{k n_2-1}{n_1}\rceil$ and hence $\sigma(G \circ H)\geq
kn_2-1+\ell-\lceil\frac{k n_2-1}{n_1}\rceil$.

$(3)$ Since $\sigma(G)=k$, there exist $k$ spanning tree in $G$, say
$T_1,T_2,\cdots,T_k$. Then there exist parallel subgraphs
$\mathscr{F}_{i,j} \ (1\leq i\leq k, 1\leq j\leq n_2)$ in $G \circ
H$ corresponding to the spanning tree $T_i \ (1\leq i\leq k)$ in
$G$. We pick up $\mathscr{F}_{k,n_2}$. Note that
$\mathscr{F}_{k,n_2}=\bigcup_{v_i\in V(H)}T_{k}(v_i)$, where
$T_{k}(v_i)$ is the corresponding tree of $T_{k}$ in $G(v_i) \ 1\leq
i\leq n_2$. Thus we can obtain $n_2$ such trees isomorphic to the
spanning tree $T_{k}$ of $H$ from $\mathscr{F}_{k,n_2}$, say
$T_{k}(v_1),T_{k}(v_2),\cdots,T_{k}(v_{n_2})$. Since $\tau(H)=\ell$,
there exist $\ell$ spanning tree in $G$, say
$T'_1,T'_2,\cdots,T'_{\ell}$. Then there exist parallel subgraphs
$\mathscr{F}_j'=\bigcup_{u_i\in V(G)}T_j'(u_i) \ (1\leq j\leq \ell)$
in $G \circ H$ corresponding to the spanning tree $T_j'$ in $H$,
where $T'_j(u_i)$ is the corresponding tree of $T_j'$ in $H(u_i)$.
Note that one of
$\mathscr{F}_1',\mathscr{F}_2',\cdots,\mathscr{F}_{\ell}'$ and one
of $T_{k}(v_1),T_{k}(v_2),\cdots,T_{k}(v_{n_2})$ can form a spanning
tree of $G\Box H$. Since $\ell\leq \lfloor\frac{n_2}{2}\rfloor$, we
can obtain $\ell$ spanning trees of $G\Box H$. Note that we also
have $kn_2-1$ parallel subgraphs $\{\mathscr{F}_{i,j}|1\leq i\leq k,
1\leq j\leq n_2\}\setminus \mathscr{F}_{k,n_2}$. Pick up $2x$
parallel subgraphs from $\{\mathscr{F}_{i,j}|1\leq i\leq k, 1\leq
j\leq n_2\}\setminus \mathscr{F}_{k,n_2}$, say
$\mathscr{F}_{a_1,b_1},\mathscr{F}_{a_2,b_2},\cdots,\mathscr{F}_{a_{2x},b_{2x}}
\ (a_{1},a_2,\cdots,a_{2x}\in \{1,2,\cdots,k\})$, such that
$\mathscr{F}_{a_{2r-1},b_{2r-1}}\cup \mathscr{F}_{a_{2r},b_{2r}} \
(1\leq r\leq x)$ contains $(n_1-1)$ perfect cycles. So we can obtain
$x(n_1-1)$ perfect cycles from the above $2x$ parallel subgraphs.
Now we still have $kn_2-1-2x$ parallel subgraphs. Note that one
parallel subgraph and one perfect cycle can form a spanning subgraph
of $G\Box H$ containing a spanning tree of $G\Box H$. If
$x(n_1-1)\geq kn_2-1-2x$, then we can obtain $kn_2-1-2x$ spanning
trees of $G\Box H$. So the total number of the spanning trees of
$G\Box H$ is $(kn_2-1-2x)+\ell$. Since $x(n_1-1)\geq kn_2-1-2x$, it
follows that $x\geq \frac{kn_2-1}{n_1+1}$. We hope that $x$ is as
small as possible. So $\sigma(G\circ H)\geq \sigma(G\Box H)\geq
kn_2-1-\lceil\frac{kn_2-1}{n_1+1}\rceil+\ell
=kn_2+\ell-1+\lceil\frac{kn_2-1}{n_1+1}\rceil$.
\end{proof}

To show the sharpness of the above lower bounds, we consider the
following three examples.

\noindent\textbf{Example 2}. Let $G$ and $H$ be two connected graphs
which can be decomposed into exact $k$ and $\ell$ spanning trees of
$G$ and $H$, respectively. From $(1)$ of the above theorem,
$\sigma(G\circ H)\geq k n_2(=\ell n_1)$. Since $|E(G\circ
H)|=|E(H)|n_1+|E(G)|n_2^2=\ell(n_2-1)n_1+k(n_1-1)n_2^2
=kn_2(n_2-1)+k(n_1-1)n_2^2=kn_2(n_1n_2-1)$, we have $\sigma(G \circ
H)\leq \frac{|E(G \circ H)|}{n_1n_2-1}=kn_2$. Then $\sigma(G \circ
H)=kn_2(=\ell n_1)$. So the upper bound of $(1)$ is sharp.

\noindent\textbf{Example 3}. Consider the graphs $G=P_3$ and
$H=K_4$. Clearly, $\sigma(G)=k=1$, $\sigma(H)=\ell=2$, $n_1=3$,
$n_2=4$, $|E(G)|=2$, $|E(H)|=6$ and $6=\ell n_1>k n_2=4$. On one
hand, we have $\ell n_1-k n_2=2$ and $\tau(G \circ H)\geq
kn_2-\lceil\frac{k
n_2-1}{n_1}\rceil+\ell-1=4-1+2-\lceil\frac{4-1}{3}\rceil=4$. On the
other hand, $|E(G \circ H)|=50$. Then $\sigma(G \circ H)\leq
\frac{|E(G \circ H)|}{n_1n_2-1}=\lfloor\frac{50}{11}\rfloor=4$. So
$\sigma(G \circ H)=4$. So the upper bound of $(2)$ is sharp.

\noindent\textbf{Example 4}. Let $G=K_{4}^-$ be a graph obtained
from $K_4$ by deleting one edge, and $H=P_3$. Clearly,
$\sigma(G)=k=2$, $\sigma (H)=\ell=1$, $n_1=4$, $n_2=3$, $|E(G)|=5$,
$|E(H)|=2$ and $4=\ell n_1<k n_2=6$. On one hand, $\sigma(G \circ
H)\geq kn_2+\ell-1-2\lceil\frac{kn_2-1}{n_1+1}\rceil=4$. On the
other hand, $|E(G \circ H)|=|E(H)|n_1+|E(G)|n_2^2=53$. Then
$\sigma(G \circ H)\leq \frac{|E(G \circ
H)|}{n_1n_2-1}=\lfloor\frac{53}{11}\rfloor=4$. So $\sigma(G \circ
H)=4$ and the upper bound of $(3)$ is sharp.

\end{document}